\documentclass[smallextended]{svjour3}

\smartqed

\usepackage{verbatim,a4wide}
\usepackage{amsmath}
\usepackage{amssymb}
\usepackage{amsthm}
\usepackage[a4paper,left=2.5cm,right=2.5cm,top=2cm,bottom=3cm]{geometry}
\usepackage{graphicx}
\usepackage{caption}
\usepackage{subfig}
\usepackage{grffile}
\usepackage{float}
\usepackage{bbm}
\usepackage{indentfirst}
\usepackage{booktabs}
\usepackage{multirow}
\usepackage{rotating}
\usepackage[T1,hyphens]{url}
\usepackage{hyperref}
\usepackage{epstopdf}

\usepackage[cp1250]{inputenc}
\usepackage[OT4]{fontenc}
\usepackage[english]{babel}

\newtheorem{fact}[theorem]{Fact}

\numberwithin{equation}{section}
\numberwithin{theorem}{section}
\numberwithin{problem}{section}
\numberwithin{lemma}{section}
\numberwithin{corollary}{section}
\numberwithin{remark}{section}
\numberwithin{example}{section}
\numberwithin{proposition}{section}

\makeatletter
\let\c@proposition\c@theorem
\let\c@corollary\c@theorem
\let\c@remark\c@theorem
\let\c@problem\c@theorem
\let\c@lemma\c@theorem
\let\c@definition\c@theorem
\let\c@example\c@theorem
\makeatother

\newcommand{\ra}[1]{\renewcommand{\arraystretch}{#1}}

\newcommand{\N}{\mathbb N}
\newcommand{\R}{\mathbb R}

\newcommand{\der}{\mbox{${\rm\,d}$}}
\newcommand{\dt}{\mbox{${\rm\,d}t$}}

\begin{document}

\title{Dual polynomial spline bases}

\author{Przemys{\l}aw Gospodarczyk        \and
        Pawe{\l} Wo\'{z}ny
}

\institute{P. Gospodarczyk (Corresponding author) \at Institute of Computer Science, University of Wroc{\l}aw,
              ul.~F.~Joliot-Curie 15, 50-383 Wroc{\l}aw, Poland\\
              Fax: +48713757801\\
              \email pgo@ii.uni.wroc.pl
           \and P. Wo\'{z}ny \at Institute of Computer Science, University of Wroc{\l}aw,
                ul.~F.~Joliot-Curie 15, 50-383 Wroc{\l}aw, Poland\\
                \email Pawel.Wozny@cs.uni.wroc.pl
}

\date{\today}

\maketitle

\begin{abstract}
In the paper, we give methods of construction of dual bases for the B-spline basis and truncated power basis. Explicit formulas for the dual B-spline basis
are obtained using the Legendre-like orthogonal basis of the polynomial spline space presented in (Wei et al., Comput.-Aided Des.~45 (2013), 85--92) and a connection between orthogonal and dual bases of any space given in (Lewanowicz and Wo\'zny, J.~Approx.~Theory 138 (2006), 129--150). Construction of the dual truncated power basis is performed in two phases. We start with explicit formulas for the dual power basis of the space of polynomials. Then, we expand this basis using an iterative algorithm proposed in (Wo\'zny, J.~Comput.~Appl.~Math.~260 (2014), 301--311). As a result, we obtain the dual truncated power basis. We also present some applications of the proposed dual polynomial spline bases and illustrative examples.


\end{abstract}

\section{Introduction}\label{Sec:Intro}

Recently, \textit{dual bases} with their applications in \textit{numerical analysis} and \textit{computer graphics} have been
extensively studied. For example, there are many articles concerning the \textit{dual Bernstein polynomials} (see, e.g., \cite{Cie87,Jue98})
and their applications in several important algorithms associated with \textit{B\'ezier curves} (see, e.g., \cite{BJ07,GLW15,GLW17,LWK12,Liu09,WGL15,WL09}). Moreover, one can find some general results on dual bases (see \cite{GW15,Ker13,Woz13,Woz14}).
Nevertheless, the number of articles concerning \textit{dual polynomial spline bases} is still very limited.
In \cite{Woz14}, an iterative method of construction of the \textit{dual B-spline basis} was presented.

There are several important approximation problems concerning polynomial spline curves, e.g.,
\textit{degree reduction} (see, e.g.,~\cite{Hos87,LP01,PW09,PT95,WWF98,Woz14,Yong01}), \textit{merging} (see, e.g., \cite{CW10,THH03}) and
\textit{knots removal} (see, e.g., \cite{Eck95,LM87,Woz14}). In order to solve these problems using \textit{least squares approximation}, it is helpful to know the dual basis for the selected polynomial spline basis (see \cite[\S4]{Woz14}).

The first goal of the paper is to give a new method of construction of the dual B-spline basis.
In contrast to the \textit{iterative method} presented in \cite{Woz14}, our aim is to provide \textit{explicit formulas} for the dual B-spline functions.
Another difference is that the previous approach is based on a more general algorithm of construction of dual basis with respect to an arbitrary
inner product, whereas the new method is fully adjusted to the B-spline basis, and the $L_2$ inner product on the interval $[0,\,1]$ is selected.
The idea is to use a certain connection between orthogonal and dual bases of any space. As we shall see in \S\ref{Sec:PrelimDual}, a \textit{well-expressed orthogonal basis} of a polynomial spline space is needed. There are several papers dealing with the construction of orthogonal bases of various polynomial spline spaces (see, e.g.,~\cite{FHRS96,KTM88,MRS93,Red12,Sab88,Sch75,WGY13}). In contrast to the other methods, the one given in \cite{WGY13} provides a well-expressed orthogonal basis of the polynomial spline space defined over an arbitrary \textit{knot vector} (see \S\ref{Sec:PrelimSpln}). Consequently, it is used in our method of construction of the dual B-spline basis. Moreover, we have noted that in \cite{WGY13}, there are no formulas connecting the B-spline basis and the presented orthogonal basis. As it turns out, such formulas lead to the dual B-spline basis, and they are useful in solving least squares approximation problems concerning B-spline curves. For details, see \S\ref{Sec:DBSpl}.

The second goal of the paper is to construct the dual basis for the \textit{truncated power basis} of the polynomial spline space. The idea is to start with
the construction of the dual basis for the \textit{power basis} of the space of polynomials. We give explicit formulas for those \textit{dual power functions}. Then, in order to obtain the \textit{dual truncated power basis}, we can expand the dual power basis using the iterative algorithm from \cite{Woz14}. For details, see \S\ref{Sec:DTrunc}. According to the authors' knowledge, methods of construction of the dual power basis and dual truncated power basis have never been published before.

In \S\ref{Sec:App}, we explain the applications of the proposed dual polynomial spline bases in degree reduction and knots removal for spline curves. We also give some illustrative examples.

\section{Preliminaries}\label{Sec:Prelim}

\subsection{Dual and orthogonal bases}\label{Sec:PrelimDual}

Let $B_n := \left\{b_0,b_1,\ldots,b_{n}\right\}$ be a basis of the linear space $\mathcal{B}_n := \mbox{span}\,B_n$.  The dual basis
$D_n := \left\{d_0^{(n)},d_1^{(n)},\ldots,d_{n}^{(n)}\right\}$ for the basis $B_n$ of the space $\mathcal{B}_n$ satisfies the following
\textit{duality conditions}:
\begin{align*}
&\mbox{span}\,D_n=\mathcal{B}_n,\\[1ex]
&\hspace*{-1mm}\left<b_i,d_j^{(n)}\right>=\delta_{i,j}\qquad (i,j = 0,1,\ldots,n),
\end{align*}
where $\left<\cdot,\cdot\right>\,:\,\mathcal{B}_n\times \mathcal{B}_n\rightarrow\R$ is an inner product,
and $d_j^{(n)}$ $(j=0,1,\ldots,n)$ are called \textit{dual functions}. Here $\delta_{i,j}$, known as \textit{Kronecker delta}, equals $1$
if $i = j$, and $0$ otherwise.

Let $q_0,q_1,\ldots,q_n$ be an orthogonal basis of the space $\mathcal{B}_n$ with respect to the inner product $\left<\cdot,\cdot\right>$, i.e., the following conditions are satisfied:
\begin{align*}
&\mbox{span}\,\{q_0,q_1,\ldots,q_n\}=\mathcal{B}_n,\\[2ex]
&\hspace*{-1mm}\left<q_i,q_j\right>=h_i\delta_{i,j}\qquad (h_i > 0;\,i,j = 0,1,\ldots,n).
\end{align*}

Recall that orthogonal and dual bases of any space are related.
\begin{lemma}[\cite{LW06}]\label{L:Lew06}
Suppose that we know the representation of each $q_0,q_1,\ldots,q_n$ in the basis $b_0,b_1,\ldots,b_n$,
\begin{equation*}
q_i = \sum_{j=0}^{n}a_{i,j}b_j \qquad (i = 0,1,\ldots,n).
\end{equation*}
Then, the dual functions $d_0^{(n)},d_1^{(n)},\ldots,d_n^{(n)}$ can be written in the following way:
$$
d_j^{(n)} = \sum_{i=0}^{n}h_i^{-1}a_{i,j}q_i \qquad (j = 0,1,\ldots,n).
$$
\end{lemma}

Now, we recall the well-known fact about the use of dual and orthogonal bases in least squares approximation.
\begin{fact}[e.g., \cite{Woz13} and \protect{\cite[Theorem 4.5.13]{DB08}}]\label{T:DualLS}
The element
\begin{equation*}
p_n^\ast := \sum_{i=0}^n \left<f,d_i^{(n)}\right>b_i = \sum_{i=0}^n \frac{\left<f,q_i\right>}{\left<q_i,q_i\right>}q_i
\end{equation*}
is the best least squares approximation of a function $f$ in the space $\mathcal{B}_n$, i.e.,
$$
\|f-p_n^\ast\|=\min_{p_n\in\mathcal{B}_n}\|f-p_n\|,
$$
where $\|\cdot\|:=\sqrt{\left<\cdot,\cdot\right>}$ is the least squares norm.
\end{fact}
\noindent According to Fact \ref{T:DualLS}, the use of an orthogonal basis in the least squares approximation implies that the resulting $p_n^\ast$
will be written in this orthogonal basis. If a different representation for $p_n^\ast$ is required, then additional conversion
is necessary. Observe that an appropriate dual basis is more convenient in this regard because we get the resulting $p_n^\ast$ in the chosen basis,
and the additional conversion is not required. Moreover, the dual basis approach may lead to some algorithms having lower computational complexity
(see, e.g., \cite{GLW15,GLW17,WGL15,WL09}).

\subsection{Polynomial spline bases}\label{Sec:PrelimSpln}

Let $\Omega_n[T_m]$ be the space of \textit{polynomial spline curves} of degree at most $n$ defined over the knot vector
\begin{align}
T_m &:= \big(\underbrace{0,0,\ldots,0}_{n+1} < t_1 \le t_2 \le \cdots \le t_m <\underbrace{1,1,\ldots,1}_{n+1}\big)\label{Eq:KV}\\[0.5ex]
    &\equiv  \left(t_{-n}^{(m)}, t_{-n+1}^{(m)},\ldots, t_{n+m+1}^{(m)}\right),\nonumber
\end{align}
where $t_1,t_2,\ldots,t_m$ are called \textit{interior knots}. We assume that the multiplicity of an interior knot
does not exceed $n$. The (normalized) B-spline basis of the space $\Omega_n[T_m]$,
\begin{equation}\label{EQ:bas}
N_{-n,n}^m,N_{-n+1,n}^m,\ldots, N_{m,n}^m,
\end{equation}
can be defined recursively,
\begin{align*}
&N_{i,0}^m(t) := \left\{ \begin{array}{ll}
   1        \qquad &\left(t_{i}^{(m)} \leq t < t_{i+1}^{(m)}\right),\\[1ex]
   0        &\text{otherwise},
  \end{array}\right.\\[1ex]
&N_{i,n}^m(t) := \frac{t-t_i^{(m)}}{t_{i+n}^{(m)}-t_i^{(m)}}N_{i,n-1}^m(t)+\frac{t_{i+n+1}^{(m)}-t}{t_{i+n+1}^{(m)}-t_{i+1}^{(m)}}N_{i+1,n-1}^m(t);
\end{align*}
and explicitly,
$$
N_{i,n}^m(t) := \left(t_{i+n+1}^{(m)} - t_i^{(m)}\right)\left[t_i^{(m)},t_{i+1}^{(m)},\ldots, t_{i+n+1}^{(m)}\right](\cdot - t)^n_+,
$$
where
$$
(x - t)^n_+ := \left\{ \begin{array}{ll}\displaystyle (x-t)^n \qquad &(x > t),\\[1ex]
                                        \displaystyle  0 &(x \leq t).\end{array} \right.
$$
Here $\left[x_1,x_2,\ldots,x_k\right]f$ is a \textit{divided difference} (see, e.g., \cite[\S4.2.1]{DB08}).
At an interior knot $N_{i,n}^m(t)$ is $n-k$ times continuously differentiable, where $k$ is the multiplicity of the interior knot.
For some other useful properties of the B-spline basis, we recommend \cite[\S2]{PT97}.

In the paper, $\Pi_n$ denotes the space of all polynomials of degree at most $n$. Observe that $\Pi_n \equiv \Omega_n[T_0]$ on $[0,1]$, where
\begin{equation*}
T_0 := \big(\underbrace{0,0,\ldots,0}_{n+1},\underbrace{1,1,\ldots,1}_{n+1}\big).
\end{equation*}
The well-known power basis
\begin{equation}\label{EQ:Trunc1}
1,t,\ldots,t^n
\end{equation}
of the space $\Pi_n$, and the functions
\begin{equation*}
(t - t_q)^{n-h}_+ \qquad (h=0,1,\ldots,l_q-1)
\end{equation*}
for each interior knot $t_q$ of multiplicity $l_q$ $(1 \le l_q \le n)$ in $T_m$ form
the truncated power basis of the space $\Omega_n[T_m]$. For example, let us assume that there is only
one multiple interior knot $t_r$, i.e., we consider
$$
T'_m := \big(\underbrace{0,0,\ldots,0}_{n+1} < t_1 < t_2 < \cdots < t_r = t_{r+1} = \cdots = t_{r+l_r-1} < t_{r+l_r} < \cdots < t_m <\underbrace{1,1,\ldots,1}_{n+1}\big).
$$
Then, the functions \eqref{EQ:Trunc1},
\begin{align*}
&(t - t_j)^{n}_+ \qquad (j=1,2,\ldots,r-1,r+l_r,\ldots,m),\\
&(t - t_r)^{n-h}_+ \qquad (h=0,1,\ldots,l_r-1)
\end{align*}
form the truncated power basis of the space $\Omega_n[T'_m]$.

Further on in the paper, the inner product $\left<\cdot,\cdot\right>_L\,:\,\Omega_n[T_m]\times \Omega_n[T_m]\rightarrow\mathbb{R}$ is defined as follows:
\begin{equation}\label{EQ:innerp}
\langle f,g\rangle_L := \int_0^1 f(t)g(t) \dt.
\end{equation}

In \S\ref{Sec:DBSpl}, we will need a well-expressed orthogonal basis of the space $\Omega_n[T_m]$ with respect to the
inner product \eqref{EQ:innerp} (cf.~Lemma~\ref{L:Lew06}). Therefore, we recall the formulas given in \cite{WGY13}.
\begin{lemma}[\cite{WGY13}]\label{L:WGY13}
\emph{Legendre-like orthogonal basis} of the space $\Omega_n[T_m]$ with respect to the inner product \eqref{EQ:innerp}
is given by
\begin{align}
&L_{i,n}(t) := \frac{1}{i!}\frac{\der^i(1-t)^it^i}{\dt^i} \qquad (i=0,1,\ldots,n),\label{EQ:2}\\[1ex]
&L_{n+k,n}(t) := \frac{\der^{n+1}}{\dt^{n+1}}f_k(t) \qquad (k=1,2,\ldots,m),\label{EQ:3}
\end{align}
where
\begin{equation*}
f_k(t) := \sum_{j=0}^{k-1} \alpha_{-n+j,k}N_{-n+j,2n+1}^k(t)
\end{equation*}
with $\alpha_{-n+j,k}$ as defined in \cite[(2) and (3)]{WGY13}.
\end{lemma}

\textit{Legendre polynomials} $P_i$ $(i=0,1,\ldots)$, which are a special case of \textit{Jacobi polynomials}, can be represented using the \textit{hypergeometric function} (see, e.g., \cite[p.~117]{AAR99}),
\begin{equation}\label{EQ:Leg}
P_i(x) = (-1)^i{}_2F_1\left({\textstyle -i,\,i+1 \atop \textstyle 1}\,\middle|\,\frac{1+x}{2}\right),
\end{equation}
where
\begin{equation}\label{Eq:hyper}
{}_2F_1\left({\textstyle -s,\,a \atop \textstyle b}\,\middle|\,t\right) := \sum_{k=0}^{s}\frac{(-s)_k(a)_k}{(b)_kk!}t^k \qquad (s \in \N)
\end{equation}
with
$$
(h)_0 := 1, \qquad (h)_k := h(h+1)\cdots(h+k-1).
$$
The functions
\begin{equation}\label{EQ:Ident}
L_{i,n}(t) \equiv (-1)^iP_i(2t-1) \qquad (i=0,1,\ldots,n)
\end{equation} are called \textit{shifted Legendre polynomials} (cf.~\eqref{EQ:2}).

\section{Dual B-spline basis}\label{Sec:DBSpl}

\subsection{Explicit formulas for the dual B-spline basis}\label{SubSec:ExpDBS}

In this section, our goal is to construct the dual basis for the B-spline basis \eqref{EQ:bas} of the space $\Omega_n[T_m]$ with respect to the
inner product \eqref{EQ:innerp}. The idea is to use Lemmas \ref{L:Lew06} and \ref{L:WGY13}. To do so, we will represent each $L_{j,n}$ $(j=0,1,\ldots,n+m)$
(see \eqref{EQ:2} and \eqref{EQ:3}) in the B-spline basis \eqref{EQ:bas} of the space $\Omega_n[T_{m}]$. Note that the given formulas are explicit.

\begin{lemma}\label{L:funL}
The shifted Legendre polynomials $L_{i,n}$ $(i=0,1,\ldots,n)$ (see \eqref{EQ:2}) can be written in terms of the B-spline functions
of degree $n$ over the knot vector $T_m$ (see \eqref{EQ:bas}),
\begin{align}
&L_{0,n}(t) = \sum_{j=0}^{n+m}N_{-n+j,n}^{m}(t),\label{EQ:L0}\\[1ex]
&L_{i,n}(t) = \sum_{j=0}^{n+m}\phi_{j,i}N_{-n+j,n}^{m}(t) \qquad (i=1,2,\ldots,n),\label{EQ:powerL}
\end{align}
where
\begin{equation}\label{Eq:polar}
\phi_{j,i} := \sum_{h=0}^{i}\frac{(-i)_h(i+1)_h}{\binom{n}{h}(h!)^2}\sum_{\substack{S \subset A_j \\ |S| = h}}\;\prod_{s \in S} t_s^{(m)}
\end{equation}
with $A_j := \left\{-n+j+1, -n+j+2,\ldots, j\right\}$.
\end{lemma}
\begin{proof}
First, it is well known that the B-spline functions form a \textit{partition of unity}, the formula \eqref{EQ:L0} is thus proven.
Next, we obtain
\begin{equation}\label{Eq:LPow}
L_{i,n}(t) = \sum_{h=0}^{i}(-i)_h(i+1)_h(h!)^{-2}t^h \qquad (i=1,2,\ldots,n)
\end{equation}
using \eqref{EQ:Ident}, \eqref{EQ:Leg} and \eqref{Eq:hyper}. The polynomials $L_{i,n}$ $(i=1,2,\ldots,n)$ are now written in the \textit{monomial form},
but we need to represent them in the basis \eqref{EQ:bas} of the space $\Omega_n[T_{m}]$. Such a conversion can be performed using the
\textit{polar forms} of $\binom{n}{h}t^h$ $(h=0,1,\ldots,n)$. This approach generalizes the \textit{Marsden's identity}
(see, e.g., \cite[\S5.8]{PBP02}). As a result, we obtain \eqref{EQ:powerL}.
\end{proof}

\begin{lemma}\label{L:deriv}
The functions $L_{n+k,n}$ $(k=1,2,\ldots,m)$ (see \eqref{EQ:3}) can be written in terms of the B-spline functions
of degree $n$ over the knot vector $T_m$ (see \eqref{EQ:bas}),
\begin{align}
&L_{n+k,n}(t) =  \sum_{i=0}^{n+m}\gamma_{i,k}N_{-n+i,n}^m(t) \qquad (k=1,2,\ldots,m-1),\label{EQ:Lnk}\\[1ex]
&L_{n+m,n}(t) =  \sum_{i=0}^{n+m}\beta_{i,m}N_{-n+i,n}^m(t),\label{EQ:Lnm}
\end{align}
where, for any $s_i \in \left[t_{-n+i}^{(m)},\,t_{i+1}^{(m)}\right)$,
\begin{align*}
& \gamma_{i,k} := \sum_{h=0}^{n+k}\beta_{h,k}\left(t^{(k)}_{h+1}-t^{(k)}_{-n+h}\right)\left[t^{(k)}_{-n+h},t^{(k)}_{-n+h+1},\ldots,t^{(k)}_{h+1}\right]g_{i},\\[1ex]
& g_{i}(y) := \left(y-s_i\right)_{+}^0\prod_{r=1}^{n}\left(y-t^{(m)}_{-n+i+r}\right),\\[1ex]
& \beta_{h,k} := \frac{(2n+1)!}{n!}\sum_{j=0}^{k-1}\alpha_{-n+j,k}\lambda_{n+1,h-j}^{(j,k)}
\end{align*}
with $\lambda_{0,0}^{(j,k)} := 1$ and
\begin{equation*}
\lambda_{v,q}^{(j,k)} := \frac{\lambda_{v-1,q}^{(j,k)} - \lambda_{v-1,q-1}^{(j,k)}}{t_{n+j+q-v+2}^{(k)} - t_{-n+j+q}^{(k)}} \qquad (v=1,2,\ldots,n+1;\; q=0,1,\ldots,v).
\end{equation*}
Here we assume that $\lambda_{u,w}^{(j,k)} =0$, if $w < 0$ or $w > u$. Moreover, if the denominator is $0$, then the corresponding quotient is defined to be $0$.
\end{lemma}
\begin{proof}
First, we substitute
$$
\frac{\der^{n+1}}{\dt^{n+1}}N_{-n+j,2n+1}^k(t) = \frac{(2n+1)!}{n!}\sum_{h=0}^{n+1}\lambda^{(j,k)}_{n+1,h}N_{-n+j+h,n}^k(t) \qquad (j=0,1,\ldots,k-1)
$$
(see, e.g., \cite[(2.10)]{PT97}) into \eqref{EQ:3}. Then, some simple manipulations lead to the formulas
\begin{equation*}
L_{n+k,n}(t) =  \sum_{h=0}^{n+k}\beta_{h,k}N_{-n+h,n}^k(t) \qquad (k=1,2,\ldots,m).
\end{equation*}
The functions $L_{n+k,n}$ $(k=1,2,\ldots,m)$  are now written in terms of the B-spline functions
of degree $n$ over the knot vector $T_k$ (cf.~\cite[Lemma 4.1]{WGY13}). The formula \eqref{EQ:Lnm} is thus proven. Next, taking into account that $\Omega_n[T_{k}] \subset \Omega_n[T_{m}]$ $(k=1,2,\ldots,m-1)$ (see, e.g., \cite[\S5.2]{PT97}), we note that each $L_{n+k,n} \in \Omega_n[T_{k}]$ $(k=1,2,\ldots,m-1)$ can be represented in the basis \eqref{EQ:bas} of the space $\Omega_n[T_{m}]$. Such a process is called \textit{knot refinement} (see, e.g., \cite[\S5.3]{PT97}), and it can be done by multiple application of \textit{knot insertion} (see, e.g., \cite[\S5.2]{PT97}). However, there are methods that specialize in knot refinement. An application of the one given in \cite{Coh80} results in the formulas \eqref{EQ:Lnk}.
\end{proof}

The final result, given in Theorem \ref{T:Main}, follows from Lemmas \ref{L:Lew06}, \ref{L:WGY13}, \ref{L:funL} and \ref{L:deriv}.

\begin{theorem}\label{T:Main}
The functions
\begin{equation}\label{Eq:dualBs}
D_{-n+j,n}^m(t) :=  \sum_{i=0}^{n+m}\varrho_{j,i}L_{i,n}(t) \qquad (j=0,1,\ldots,n+m),
\end{equation}
where
$$
\varrho_{j,i} := \left<L_{i,n},\,L_{i,n}\right>_L^{-1}\psi_{j,i}
$$
with
$$
\psi_{j,i} := \left\{ \begin{array}{ll}\displaystyle 1 \qquad &(i=0),\\[2ex]
                                        \displaystyle  \phi_{j,i} &(i=1,2,\ldots,n),\\[2ex]
                                        \displaystyle  \gamma_{j,i-n} &(i=n+1,n+2,\ldots,n+m-1),\\[2ex]
                                        \displaystyle  \beta_{j,m} &(i=n+m),\\[2ex]
                                        \end{array}\right.
$$
form the dual basis for the B-spline basis \eqref{EQ:bas} of the space $\Omega_n[T_m]$ with respect to the inner product \eqref{EQ:innerp}, i.e.,
$$
\left<N_{-n+i,n}^m,\,D_{-n+j,n}^m\right>_L = \delta_{i,j} \qquad (i,j=0,1,\ldots,n+m).
$$
\end{theorem}

\subsection{Additional remarks}\label{SubSec:DualRem}

\begin{remark}\label{R:polar}
Implementation of \eqref{Eq:polar} requires a method of computing
\begin{equation}\label{Eq:polar2}
r_{j,h} := \sum_{\substack{S \subset A_j \\ |S| = h}}\;\prod_{s \in S} t_s^{(m)} \qquad (j=0,1,\ldots,n+m;\;h=1,2,\ldots,n),
\end{equation}
where $A_j = \left\{-n+j+1, -n+j+2,\ldots, j\right\}$. Notice that the direct use of \eqref{Eq:polar2} is computationally expensive.
However, these quantities can be computed much more efficiently. Let us assume that $j$ is fixed.
According to \textit{Vieta's formulas}, we have
$r_{j,h} = (-1)^hw_{n-h,j}$ $(h=1,2,\ldots,n)$, where $w_{i,j}$ $(i=0,1,\ldots,n-1)$ are the coefficients of the monic polynomial
\begin{equation}\label{Eq:monic}
t^n + \sum_{i=0}^{n-1} w_{i,j}t^{i} = \prod_{i=1}^{n}\left(t-t^{(m)}_{-n+j+i}\right)
\end{equation}
(see, e.g., \cite[\S6.5.1]{DB08}). Now, given the right-hand side of \eqref{Eq:monic}, our goal is to compute $w_{i,j}$ $(i=0,1,\ldots,n-1)$. This can be done with the complexity $O(n^2)$ (see, e.g., \cite[p.~371]{DB08}).
\end{remark}

\begin{remark}\label{R:innerP}
Note that the formulas \eqref{Eq:dualBs} require a method of computing the inner products
$$
\left<L_{i,n},\,L_{i,n}\right>_L = \int_0^1 \left[L_{i,n}(t)\right]^2\dt \qquad (i = 0,1,\ldots,n+m).
$$
As is known,
\begin{equation}\label{Eq:InnerL}
\left<L_{i,n},\,L_{i,n}\right>_L = (2i+1)^{-1} \qquad (i = 0,1,\ldots,n)
\end{equation}
(see, e.g., \cite[p.~568]{WL09}). Moreover, since each $L_{i,n}$ $(i=n+1,n+2,\ldots,n+m)$ is now written in the basis \eqref{EQ:bas} (see Lemma \ref{L:deriv}),
it is sufficient to use a method of computing the inner products
\begin{equation}\label{Eq:InnerN}
\left<N_{-n+i,n}^m,\,N_{-n+j,n}^m\right>_L = \int_0^1 N_{-n+i,n}^m(t)N_{-n+j,n}^m(t)\dt \qquad (i,j = 0,1,\ldots,n+m).
\end{equation}
Such methods can be found, e.g., in \cite{BB06,Ver92}.
\end{remark}

\begin{remark}
According to Fact \ref{T:DualLS}, the use of the orthogonal basis \eqref{EQ:2}, \eqref{EQ:3} in solving least squares approximation problems concerning B-spline curves results in a spline curve which is written in that basis. However, the B-spline form is much more popular in CAD systems, therefore, the resulting spline curve should be converted (cf.~\cite[\S7]{WGY13}). In \cite{WGY13}, there are no formulas that directly represent the orthogonal basis \eqref{EQ:2}, \eqref{EQ:3} in the B-spline basis \eqref{EQ:bas}. Such formulas are necessary for the conversion. In our paper, we give these formulas (see \eqref{EQ:L0}, \eqref{EQ:powerL}, \eqref{EQ:Lnk} and \eqref{EQ:Lnm}). Furthermore, observe that any B-spline curve written in the basis \eqref{EQ:bas} can be represented in the basis \eqref{EQ:2}, \eqref{EQ:3} using the formulas
\begin{equation}\label{Eq:sigma}
N_{-n+j,n}^m(t) = \sum_{i=0}^{n+m}\frac{\left<N_{-n+j,n}^m,\,L_{i,n}\right>_L}{\left<L_{i,n},\,L_{i,n}\right>_L}L_{i,n}(t) \qquad (j=0,1,\ldots,n+m).
\end{equation}
In order to compute the inner products in \eqref{Eq:sigma}, it is sufficient to deal with the inner products \eqref{Eq:InnerN}
(for explanation, see Remark \ref{R:innerP}).
\end{remark}

\section{Dual truncated power basis}\label{Sec:DTrunc}

In this section, we construct the dual basis for the truncated power basis
$$
1,t,\ldots,t^n,(t-t_1)_+^n,(t-t_2)_+^n,\ldots,(t-t_m)_+^n
$$
of the space $\Omega_n[T''_m]$, where
\begin{equation*}\label{Eq:SimpleKV}
T''_m := \big(\underbrace{0,0,\ldots,0}_{n+1} < t_1 < t_2 < \cdots < t_m <\underbrace{1,1,\ldots,1}_{n+1}\big)
\end{equation*}
(cf.~\eqref{Eq:KV}), with respect to the inner product \eqref{EQ:innerp}. The procedure working for any $T_m$ \eqref{Eq:KV} is a simple generalization of the one given here. For the sake of clarity of the idea, it is omitted in the paper.

First, we use Lemma \ref{L:Lew06} and the shifted Legendre
polynomials \eqref{EQ:2} to find the dual basis for the power basis \eqref{EQ:Trunc1} of the space $\Pi_n$ with respect to the inner product \eqref{EQ:innerp} (see Lemma \ref{L:dualPower}). Then, in \S\ref{SubSec:DualTrunPow}, we obtain the dual truncated power basis by applying the iterative method from \cite{Woz14} (see also \cite{Woz13,GW15}) to the dual power basis.

\begin{lemma}\label{L:dualPower}
The polynomials
\begin{equation}\label{Eq:DualPow}
d_{j,n}(t) := \sum_{i=j}^n\Phi_{i,j}L_{i,n}(t) \qquad (j=0,1,\ldots,n),
\end{equation}
where
$$
\Phi_{i,j} := (2i+1)(-i)_j(i+1)_j(j!)^{-2},
$$
form the dual basis for the power basis \eqref{EQ:Trunc1} of the space $\Pi_n$ with respect to the inner product \eqref{EQ:innerp}, i.e.,
$$
\left<t^i,\,d_{j,n}\right>_L = \delta_{i,j} \qquad (i,j=0,1,\ldots,n).
$$
\end{lemma}
\begin{proof}
The result follows from Lemma~\ref{L:Lew06}, \eqref{Eq:LPow} and \eqref{Eq:InnerL}.
\end{proof}

\subsection{Iterative construction of the dual truncated power basis}\label{SubSec:DualTrunPow}

Let there be given the dual basis $D_n = \left\{d_0^{(n)},d_1^{(n)},\ldots,d_{n}^{(n)}\right\}$ for the basis $B_n = \left\{b_0,b_1,\ldots,b_{n}\right\}$ of the linear space $\mathcal{B}_{n} = \mbox{span}\,B_n$ with respect to an inner product $\left<\cdot,\cdot\right>$. The dual basis
$D_{n+1} = \left\{d_0^{(n+1)},d_1^{(n+1)},\ldots,d_{n+1}^{(n+1)}\right\}$ for the basis
$B_{n+1} = B_{n} \cup \left\{b_{n+1}\right\}$ of the linear space $\mathcal{B}_{n+1} = \mbox{span}\,B_{n+1}$ with respect to the same inner product can be computed efficiently using the method given in \cite{Woz14}. Our idea is to set
$$
b_i := t^i, \quad d_i^{(n)} := d_{i,n} \qquad (i=0,1,\ldots,n)
$$
(see \eqref{Eq:DualPow}), $b_{n+1} := (t-t_1)_+^n$, $\left<\cdot,\cdot\right> := \left<\cdot,\cdot\right>_L$; and apply the method. Consequently, we obtain the dual basis $D_{n+1} = \left\{d_0^{(n+1)},d_1^{(n+1)},\ldots,d_{n+1}^{(n+1)}\right\}$ for the truncated power basis $B_{n+1} = \left\{1,t,\ldots,t^n,(t-t_1)_+^n\right\}$
of the space $\Omega_n[T_1]$ with respect to the inner product \eqref{EQ:innerp}. Repeated application of the method gives
the dual basis $D_{n+m} = \left\{d_0^{(n+m)},d_1^{(n+m)},\ldots,d_{n+m}^{(n+m)}\right\}$ for the truncated power basis
$B_{n+m} = \left\{1,t,\ldots,t^n,(t-t_1)_+^n,\right.$ $\left.(t-t_2)_+^n,\ldots,(t-t_m)_+^n\right\}$ of the space $\Omega_n[T''_m]$ with respect to
the inner product \eqref{EQ:innerp}. In \S\ref{SubSubSec:First} and \S\ref{SubSubSec:ith}, we give a more detailed description of the first and $i$th $(i > 1)$ iterations, respectively.

\subsubsection{First iteration of the algorithm}\label{SubSubSec:First}

Recall that in the first iteration, $B_{n+1} = \left\{1,t,\ldots,t^n,(t-t_1)_+^n\right\}$ and
$D_n = \left\{d_{0,n},d_{1,n},\ldots,d_{n,n}\right\}$ (see \eqref{Eq:DualPow}) are given. Our goal is to compute the dual basis $D_{n+1} = \left\{d_0^{(n+1)},d_1^{(n+1)},\ldots,d_{n+1}^{(n+1)}\right\}$ for the basis $B_{n+1}$ of the space $\Omega_n[T_1]$ with respect to the inner product \eqref{EQ:innerp}.

According to \cite[(2.7)]{Woz14}, we must deal with certain inner products. First, after some algebra,
we obtain
\begin{equation}
v_j^{(n+1)} := \left<t^j,\,(t-t_1)_+^n\right>_L 
            = \frac{(1-t_1)^{n+1}t_1^j}{n+1}{}_2F_1\left({\textstyle -j,\,n+1 \atop \textstyle n+2}\,\middle|\,\frac{t_1-1}{t_1}\right)
            \qquad (j=0,1,\ldots,n)\label{EQ:v1}
\end{equation}
(see \eqref{Eq:hyper}), and
\begin{equation}
v_{n+1}^{(n+1)} := \left<(t-t_1)_+^n,\,(t-t_1)_+^n\right>_L 
                = (2n+1)^{-1}(1-t_1)^{2n+1}.\label{EQ:v2}
\end{equation}
Next, using \eqref{Eq:DualPow} and \eqref{Eq:LPow}, we get
\begin{equation}
u_j^{(n+1)} := \left<(t-t_1)_+^n,\,d_{j,n}\right>_L 
            = \sum_{k=j}^n\Phi_{k,j}\sum_{h=0}^k(-k)_h(k+1)_h(h!)^{-2}v_h^{(n+1)} \qquad (j=0,1,\ldots,n).\label{EQ:u}
\end{equation}
Finally, thanks to \cite[Algorithm 2.5]{Woz14}, we obtain the following representation of dual functions:
\begin{equation}\label{EQ:djn}
d_{j}^{(n+1)}(t) =  \sum_{k=0}^n \Psi_{k,j}^{(1)}L_{k,n}(t) + \Lambda_{j}^{(1)}(t-t_1)_+^n \qquad (j=0,1,\ldots,n+1),
\end{equation}
where
\begin{align*}
&\Psi_{k,j}^{(1)} := \Phi_{k,j}-u_j^{(n+1)}\Psi_{k,n+1}^{(1)} \qquad (j,k=0,1,\ldots,n),\\[1ex]
&\Psi_{k,n+1}^{(1)} := \sum_{l=0}^kc_l^{(n+1)}\Phi_{k,l} \qquad (k=0,1,\ldots,n),\\[1ex]
&\Lambda_{j}^{(1)} \equiv \Lambda_{1,j}^{(1)} := -u_j^{(n+1)}\Lambda_{n+1}^{(1)}\qquad (j=0,1,\ldots,n),\\[1ex]
&\Lambda_{n+1}^{(1)} \equiv \Lambda_{1,n+1}^{(1)} := c_{n+1}^{(n+1)}
\end{align*}
with
\begin{align}
&c_h^{(q)} := -v_h^{(q)}c_{q}^{(q)} \qquad (h=0,1,\ldots,q-1),\label{EQ:ccoeff1}\\[1ex]
&c_{q}^{(q)} := \left(v_{q}^{(q)} - \sum_{h=0}^{q-1}v_{h}^{(q)}u_{h}^{(q)}\right)^{-1}.\label{EQ:ccoeff2}
\end{align}
Here we ignore that $\Phi_{k,j} = 0$ $(k < j)$.

\subsubsection{$i$th $(i > 1)$ iteration of the algorithm}\label{SubSubSec:ith}

Let us assume that in the previous iteration of the algorithm we have computed
\begin{equation}\label{EQ:djn1}
d_{j}^{(n+i-1)}(t) = \sum_{k=0}^{n}\Psi_{k,j}^{(i-1)}L_{k,n}(t) + \sum_{h=1}^{i-1}\Lambda_{h,j}^{(i-1)}(t-t_h)_+^n \qquad (j=0,1,\ldots,n+i-1)
\end{equation}
(cf.~\eqref{EQ:djn}). Now, our goal is to determine the dual basis $D_{n+i} = \left\{d_0^{(n+i)},d_1^{(n+i)},\ldots,d_{n+i}^{(n+i)}\right\}$ for the basis
$B_{n+i}= \left\{1,t,\ldots,t^n,(t-t_1)_+^n,(t-t_2)_+^n,\ldots,(t-t_i)_+^n\right\}$ of the space $\Omega_n[T''_i]$ with respect to the inner product \eqref{EQ:innerp}.

As in the first iteration, we must deal with certain inner products. After some algebra, we get
\begin{equation}\label{EQ:vj}
v_j^{(n+i)} := \left<t^j,\,(t-t_i)_+^n\right>_L = \frac{(1-t_i)^{n+1}t_i^j}{n+1}{}_2F_1\left({\textstyle -j,\,n+1 \atop \textstyle n+2}\,\middle|\,\frac{t_i-1}{t_i}\right) \qquad (j=0,1,\ldots,n)
\end{equation}
(see \eqref{Eq:hyper}; cf.~\eqref{EQ:v1}) and
\begin{equation}
v_{n+j}^{(n+i)} := \left<(t-t_{j})_+^n,\,(t-t_i)_+^n\right>_L
            = v_{n+i}^{(n+i)}{}_2F_1\left({\textstyle -n,\,-2n-1 \atop \textstyle -2n}\,\middle|\,\frac{t_i-t_j}{t_i-1}\right)
             \qquad (j=1,2,\ldots,i-1),\label{EQ:vj2}
\end{equation}
where
\begin{equation}\label{EQ:vj3}
v_{n+i}^{(n+i)} := \left<(t-t_i)_+^n,\,(t-t_i)_+^n\right>_L = (2n+1)^{-1}(1-t_i)^{2n+1}
\end{equation}
(cf.~\eqref{EQ:v2}). Moreover, we calculate
\begin{align}
u_j^{(n+i)} &:= \left<(t-t_i)_+^n,\,d_j^{(n+i-1)}\right>_L\notag\\[1ex]
            &= \sum_{k=0}^{n}\Psi_{k,j}^{(i-1)}\sum_{q=0}^k(-k)_q(k+1)_q(q!)^{-2}v_{q}^{(n+i)} + \sum_{h=1}^{i-1}\Lambda_{h,j}^{(i-1)}v_{n+h}^{(n+i)} \qquad (j=0,1,\ldots,n+i-1)\label{EQ:q2}
\end{align}
(cf.~\eqref{EQ:u}).
As in \S\ref{SubSubSec:First}, we obtain the following representation of dual functions:
\begin{equation}\label{EQ:djni}
d_{j}^{(n+i)}(t) =  \sum_{k=0}^n \Psi_{k,j}^{(i)}L_{k,n}(t) + \sum_{h=1}^{i}\Lambda_{h,j}^{(i)}(t-t_h)_+^n \qquad (j=0,1,\ldots,n+i),
\end{equation}
where
\begin{align*}
&\Psi_{k,j}^{(i)} := \Psi_{k,j}^{(i-1)}-u^{(n+i)}_j\Psi_{k,n+i}^{(i)} \qquad (j=0,1,\ldots,n+i-1;\;k=0,1,\ldots,n),\\[1ex]
&\Psi_{k,n+i}^{(i)} :=  \sum_{l=0}^{n+i-1}c_l^{(n+i)}\Psi_{k,l}^{(i-1)} \qquad (k=0,1,\ldots,n),\\[1ex]
&\Lambda_{h,j}^{(i)} := \Lambda_{h,j}^{(i-1)}-u^{(n+i)}_j\Lambda_{h,n+i}^{(i)} \qquad (j=0,1,\ldots,n+i-1;\;h=1,2,\ldots,i-1),\\[1ex]
&\Lambda_{i,j}^{(i)} := -u^{(n+i)}_j\Lambda_{i,n+i}^{(i)} \qquad (j=0,1,\ldots,n+i-1),\\[1ex]
&\Lambda_{h,n+i}^{(i)} :=  \sum_{l=0}^{n+i-1}c_l^{(n+i)}\Lambda_{h,l}^{(i-1)} \qquad (h=1,2,\ldots,i-1),\\[1ex]
&\Lambda_{i,n+i}^{(i)} := c_{n+i}^{(n+i)}
\end{align*}
with $c_{j}^{(n+i)}$ as defined in \eqref{EQ:ccoeff1} and \eqref{EQ:ccoeff2}.

\subsection{Additional remarks}

\begin{remark}\label{R:Clenshaw}
The well-known \textit{three-term recurrence relation} for the Legendre polynomials $P_{i}$ (see, e.g., \cite[(4.5.55)]{DB08})
implies the three-term recurrence relation for the shifted Legendre polynomials $L_{i,n}$,
\begin{align*}
&L_{0,n}(t) = 1, \quad  L_{1,n}(t) = 1-2t,\\[1ex]
&L_{i+1,n}(t) =  \frac{(2i+1)(1-2t)}{i+1}L_{i,n}(t)-\frac{i}{i+1}L_{i-1,n}(t) \qquad (i = 1,2,\ldots,n-1).
\end{align*}
As a result, the linear combinations of $L_{i,n}$ $(i = 0,1,\ldots,n)$ in \eqref{Eq:DualPow}, \eqref{EQ:djn}, \eqref{EQ:djn1} and \eqref{EQ:djni}
can be evaluated with the complexity $O(n)$ using the \textit{Clenshaw's algorithm} (see, e.g., \cite[Theorem 4.5.21]{DB08}).
\end{remark}

\begin{lemma}
For a fixed $i$, the quantities $v^{(n+i)}_j$ $(j=0,1,\ldots,n)$ can be computed recursively using
\begin{align}
&v^{(n+i)}_0 = (1-t_i)^{n+1}(n+1)^{-1}, \quad v^{(n+i)}_1 = t_iv^{(n+i)}_0 + (1-t_i)^{n+2}(n+2)^{-1},\label{EQ:vr1}\\[2ex]
&v^{(n+i)}_j = \frac{((a+2)j + n(a+1))t_i}{n+j+1}v^{(n+i)}_{j-1} - \frac{(a+1)(j-1)t_i^2}{n+j+1}v^{(n+i)}_{j-2}
\qquad (j = 2,3,\ldots,n),\label{EQ:vr2}
\end{align}
where $a = (1-t_i)/t_i$, with the complexity $O(n)$ (cf.~\eqref{EQ:v1} and \eqref{EQ:vj}).
\end{lemma}
\begin{proof}
Since $v^{(n+i)}_j$ $(j=0,1,\ldots,n)$ can be written using hypergeometric functions (see \eqref{EQ:v1} and \eqref{EQ:vj}), the recurrence relation \eqref{EQ:vr2} follows from one of the \textit{Gauss' contiguous relations} (see \cite[(2.5.16)]{AAR99}).
\end{proof}

\begin{remark}
In the case of an arbitrary knot vector $T_m$ \eqref{Eq:KV}, the iterative procedure from \cite{Woz14} can be applied as well. Integrals
can be computed using formulas similar to \eqref{EQ:vj}--\eqref{EQ:q2}, \eqref{EQ:vr1} and \eqref{EQ:vr2}.
Alternatively, one can use the well-known \textit{Gaussian quadratures} (see, e.g., \cite[\S5.3]{DB08}), and get exact results since for properly chosen subintervals, integrands are polynomials.
\end{remark}

\section{Applications}\label{Sec:App}

\subsection{Degree reduction and knots removal for spline curves}

In this section, we present some applications of the proposed dual polynomial spline bases. Let us consider a degree $n$ B-spline curve in $\R^d$
\begin{equation}\label{E:curve1}
C(t) \equiv (C_1(t),C_2(t),\ldots,C_d(t)) := \sum_{i=0}^{n+m}c_iN_{-n+i,n}^m(t) \qquad (0 \le t \le 1)
\end{equation}
associated with the knot vector $T_m$ (see \eqref{Eq:KV}), where $c_i := (c_{i,1},c_{i,2},\ldots,c_{i,d}) \in \R^d$
$(i=0,1,\ldots,n+m)$.

The \textit{problem of the optimal degree reduction} of the B-spline curve $C(t)$ with respect to the $L_2$ norm is to compute the degree $n^\ast$ $(n^\ast<n)$ B-spline curve in $\R^d$
\begin{equation*}
C^\ast(t) \equiv (C_1^\ast(t),C_2^\ast(t),\ldots,C_d^\ast(t)) := \sum_{i=0}^{n^\ast+m}c_i^\ast N_{-n^\ast+i,n^\ast}^m(t) \qquad (0 \le t \le 1),
\end{equation*}
where $c_i^\ast := (c_{i,1}^\ast,c_{i,2}^\ast,\ldots,c_{i,d}^\ast) \in \R^d$
$(i=0,1,\ldots,n^\ast+m)$,
which is associated with the knot vector having the same inner knots as in $T_m$, and minimizes the $L_2$ error,
\begin{equation*}
E_2 := \|C-C^\ast\|_{L} \equiv \sqrt{\int_0^1\|C(t)-C^\ast(t)\|_E^2 \dt},
\end{equation*}
where $\|\cdot\|_E$ is the \textit{Euclidean vector norm in $\R^d$}. Clearly, the solution can be derived in a componentwise way.
According to Fact \ref{T:DualLS}, the optimal control points can be computed using the formulas
\begin{equation}\label{E:optcon}
c_{i,j}^\ast := \left<C_j,\,D_{-n^\ast+i,n^\ast}^m\right>_L \qquad (i=0,1,\ldots,n^\ast+m;\;j=1,2,\ldots,d),
\end{equation}
where $D_{-n^\ast+i,n^\ast}^m$ $(i=0,1,\ldots,n^\ast+m)$ is the dual B-spline basis for the B-spline basis $N_{-n^\ast+i,n^\ast}^m$ $(i=0,1,\ldots,n^\ast+m)$ with respect to the inner product \eqref{EQ:innerp}. In the next subsection, we also give the maximum errors,
$$
E_{\infty} := \max_{t \in W_M} \|C(t)-C^\ast(t)\|_E \approx \max_{t \in [0,1]} \|C(t)-C^\ast(t)\|_E,
$$
where $W_M := \left\{0, 1/M, 2/M,\ldots, 1\right\}$ with $M := 500$.

The \textit{problem of the optimal knots removal} for the B-spline curve $C(t)$ with respect to the $L_2$ norm is to compute the B-spline curve of the same degree which is associated with the given knot vector
\begin{equation*}
U_{m^{\ast}} := \big(\underbrace{0,0,\ldots,0}_{n+1} < u_1 \le u_2 \le \cdots \le u_{m^{\ast}} <\underbrace{1,1,\ldots,1}_{n+1}\big),
\end{equation*}
where $U_{m^{\ast}} \subset T_m$, and minimizes the $L_2$ error. Clearly, we can use the dual B-spline basis to find the optimal B-spline curve, i.e., the optimal control points can be computed using Fact \ref{T:DualLS} (cf.~\eqref{E:optcon}). Note that it is also possible to combine degree reduction with knots removal, i.e., to compute
the optimal $n^\ast$ degree B-spline curve which is associated with the knot vector $U_{m^{\ast}}$. In this case, both operations are performed simultaneously, and the dual B-spline basis is useful once again thanks to Fact \ref{T:DualLS}.

Observe that if $C(t)$ is written in the truncated power basis, then degree reduction and knots removal can be performed in the same way but using the dual truncated power basis. Furthermore, it can be easily checked that we may use the dual truncated power basis to get the representation of $C(t)$ in the truncated power basis (cf.~Fact \ref{T:DualLS}),
\begin{equation*}
C_j(t) = \sum_{i=0}^{n+m}c_{i,j}N_{-n+i,n}^m(t) = \sum_{i=0}^{n+m}\left<C_j,\,d_{i,n}^m\right>_Lp_{i,n}^m \qquad (j=1,2,\ldots,d;\;0 \le t \le 1),
\end{equation*}
where $p_{i,n}^m$ and $d_{i,n}^m$ $(i=0,1,\ldots,n+m)$ are the truncated and dual truncated power bases, respectively, of the space $\Omega_n[T_m]$ with respect to the inner product \eqref{EQ:innerp}. A spline curve in such a form can be, e.g., easily integrated.

\subsection{Examples}

Now, let us look at some examples. The results have been obtained in Maple{\small \texttrademark}13 using $16$-digit arithmetic.

The B-spline planar curve ,,Pear'' of degree $5$ is associated with the following knot vector:
\begin{equation}\label{E:K19}
K := \left(\underbrace{0,0,\ldots,0,}_{6} \frac{1}{20}, \frac{2}{20},\ldots, \frac{19}{20}, \underbrace{1,1,\ldots,1}_{6}\right)
\end{equation}
(cf.~\eqref{Eq:KV}), and its control points are given in Table \ref{tab:table1} (cf.~\eqref{E:curve1}).
\begin{table}[H]
\captionsetup{margin=0pt, font={scriptsize}}
\ra{1.3}
\begin{center}
  \begin{tabular}{lllll}
    \hline
    $(0.385, 0.845)$ & $(0.325, 0.76)$ & $(0.305, 0.635)$ & $(0.275, 0.79)$ & $(0.295, 0.895)$\\
    $(0.025, 0.885)$ & $(0.035, 0.79)$ & $(0.11, 0.71)$ & $(0.405, 0.705)$ & $(0.2, 0.675)$\\
    $(0.1, 0.59)$ & $(0.185, 0.365)$ & $(0.01, 0.45)$ & $(0.01, 0.045)$ & $(0.13, 0.02)$\\
    $(0.4, 0.005)$ & $(0.625, 0.045)$ & $(0.67, 0.185)$ & $(0.655, 0.395)$ & $(0.47, 0.405)$\\
     $(0.535, 0.51)$ & $(0.47, 0.645)$ & $(0.39, 0.71)$ & $(0.285, 0.665)$ & $(0.395, 0.835)$\\
    \hline
  \end{tabular}
\end{center}
\caption{Control points of the original B-spline planar curve ,,Pear''.}
\label{tab:table1}
\end{table}

First, let us consider the problem of knots removal for the B-spline curve ,,Pear''. We have computed the optimal B-spline curves of degree $5$
that are associated with the knot vectors
\begin{equation}\label{E:K12}
K \setminus \left\{\frac{1}{20},\frac{4}{20},\frac{7}{20},\frac{10}{20},\frac{13}{20},\frac{16}{20},\frac{19}{20}\right\}
\end{equation}
 (errors: $E_2 = 1.08e{-}2$, $E_{\infty} = 2.95e{-}2$; see Figure \ref{figure1a}) and
 \begin{equation}\label{E:K15}
K \setminus \left\{\frac{4}{20},\frac{7}{20},\frac{13}{20},\frac{16}{20}\right\}
\end{equation}
(errors: $E_2 = 3.58e{-}3$, $E_{\infty} = 7.92e{-}3$; see Figure \ref{figure1b}).

\begin{figure}[H]
\captionsetup{margin=0pt, font={scriptsize}}
\begin{center}
\setlength{\tabcolsep}{0mm}
\begin{tabular}{c}
\subfloat[]{\label{figure1a}\includegraphics[width=0.446\textwidth]{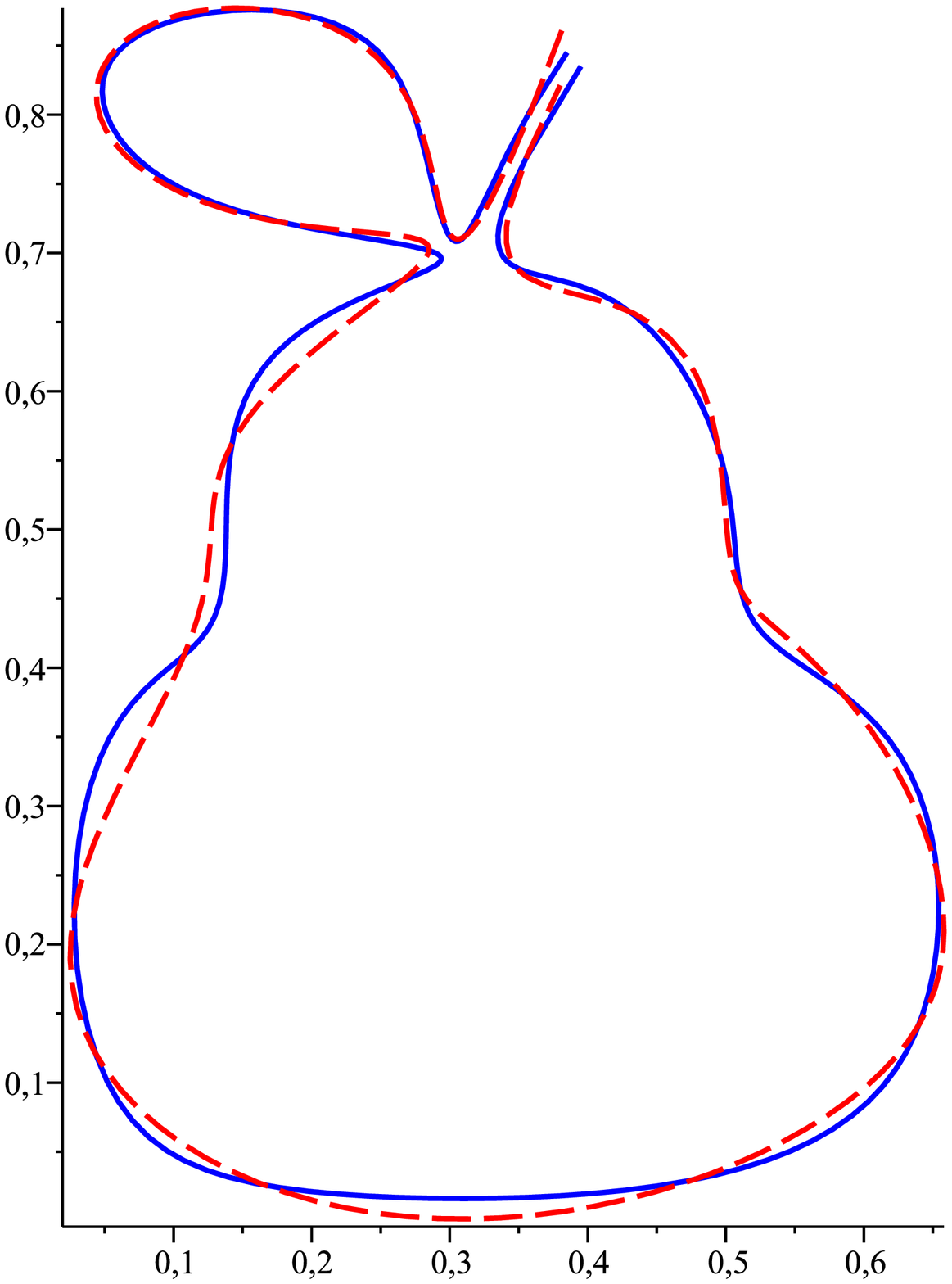}}
\subfloat[]{\label{figure1b}\includegraphics[width=0.446\textwidth]{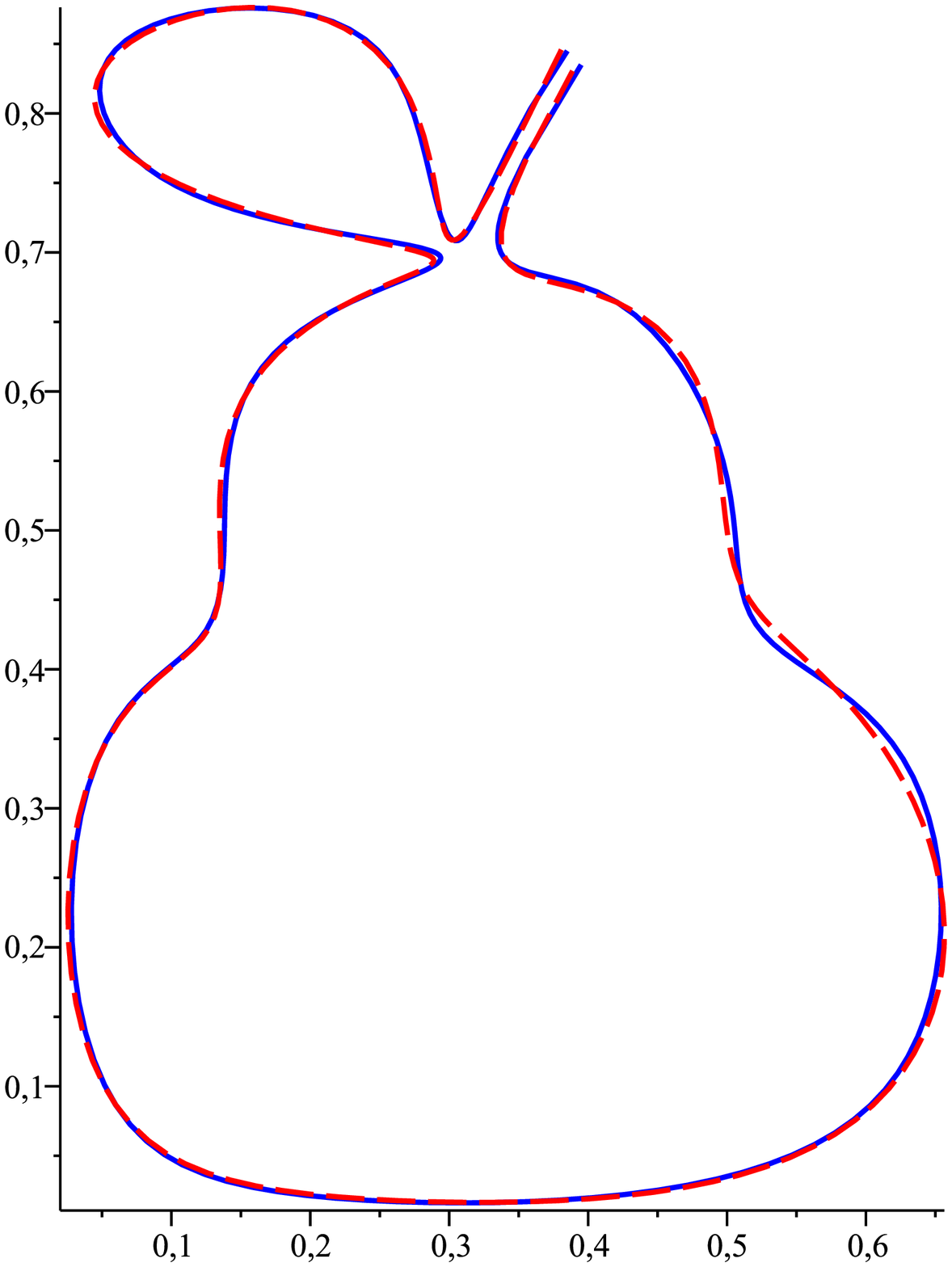}}
\end{tabular}
    \caption{The original B-spline curve ,,Pear'' of degree $5$ which is associated with the knot vector \eqref{E:K19} (blue solid line), the B-spline curves of degree $5$ which are associated with the knot vectors (a) \eqref{E:K12} and (b) \eqref{E:K15}, and computed in the process of knots removal
    (red dashed lines).}
\end{center}
\end{figure}

\newpage

Next, let us apply the procedure of degree reduction to the B-spline curve ,,Pear''. The resulting optimal B-spline curve of degree $3$
(errors: $E_2 = 2.76e{-}3$, $E_{\infty} = 3.41e{-}2$) is
shown in Figure \ref{figure2a}. The inner knots of this degree reduced curve are the same as in \eqref{E:K19}.
The optimal B-spline curve of degree $4$ (errors: $E_2 = 4.64e{-}3$, $E_{\infty} = 1.55e{-}2$) which is associated with the knot vector
\begin{equation}\label{E:K16}
K \setminus \left\{\frac{4}{20},\frac{13}{20},\frac{16}{20}\right\}
\end{equation}
is illustrated in Figure \ref{figure2b}. This curve is a result of simultaneous degree reduction and knots removal.

\begin{figure}[H]
\captionsetup{margin=0pt, font={scriptsize}}
\begin{center}
\setlength{\tabcolsep}{0mm}
\begin{tabular}{c}
\subfloat[]{\label{figure2a}\includegraphics[width=0.446\textwidth]{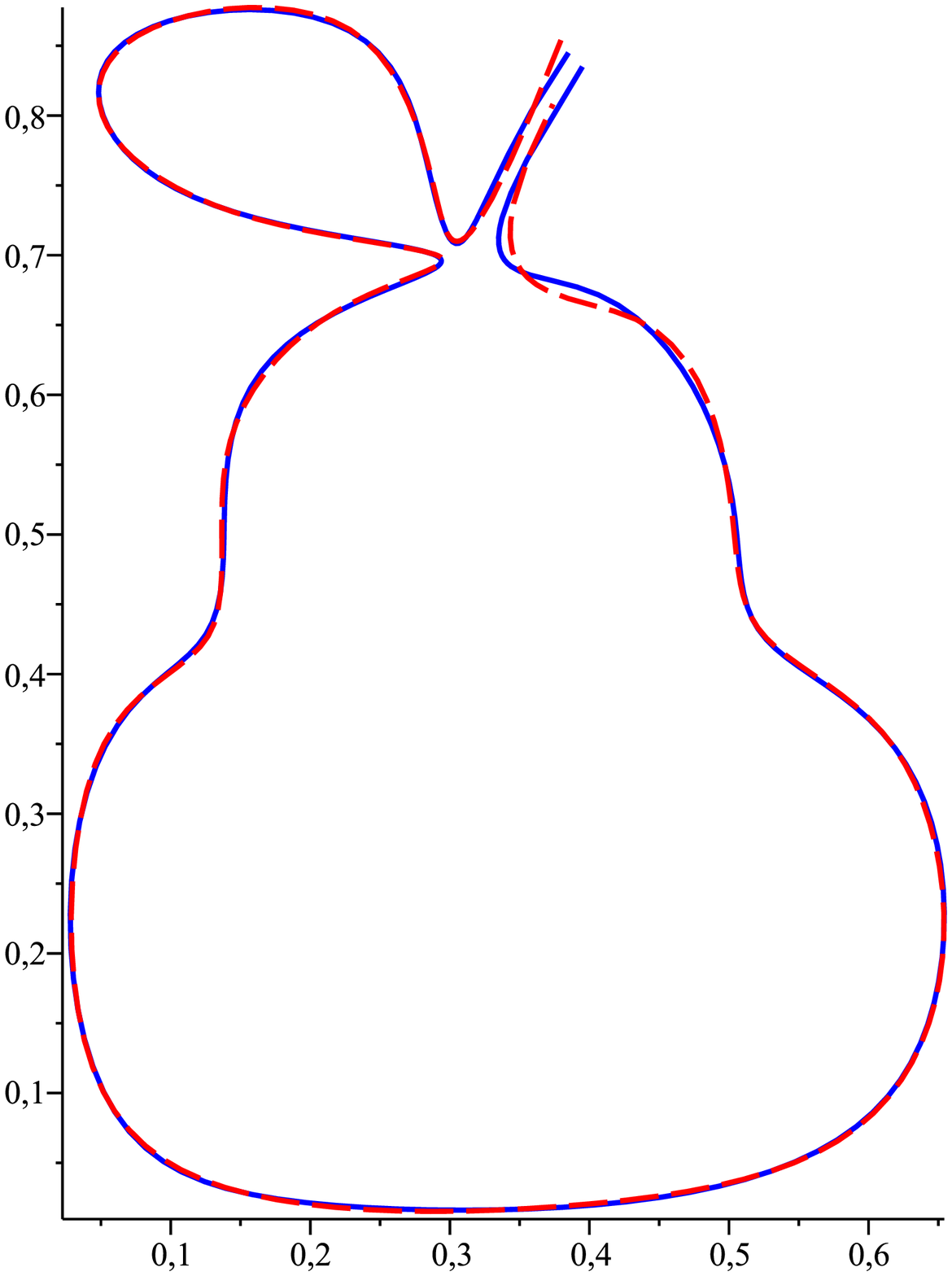}}
\subfloat[]{\label{figure2b}\includegraphics[width=0.446\textwidth]{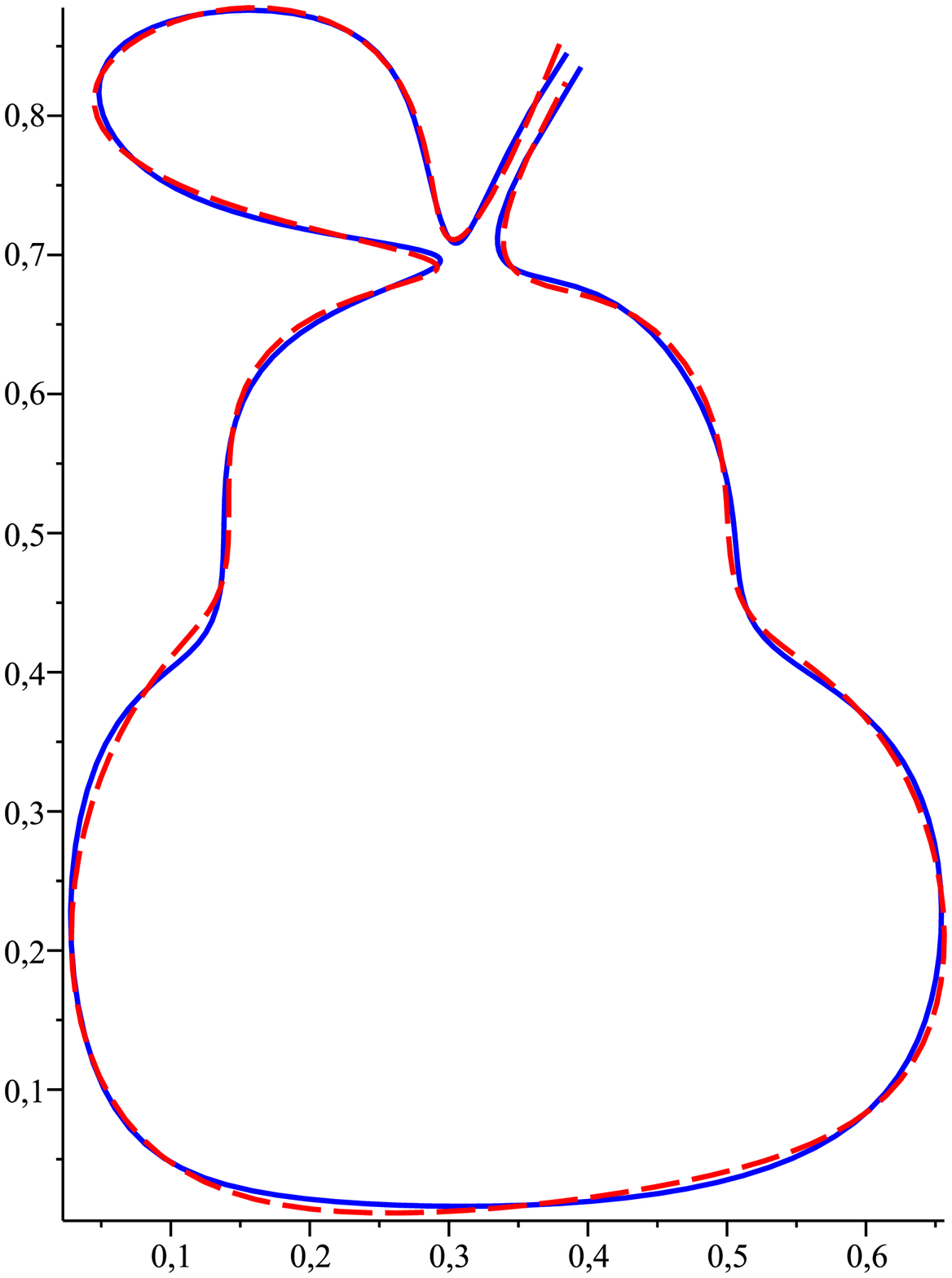}}
\end{tabular}
    \caption{The original B-spline curve ,,Pear'' of degree $5$ which is associated with the knot vector \eqref{E:K19} (blue solid line),
    (a) the B-spline curve of degree $3$ which is associated with the knot vector having the same inner knots as in \eqref{E:K19}, and
    computed in the process of degree reduction (red dashed line), (b) the B-spline curve of degree $4$ which is
    associated with the knot vector \eqref{E:K16}, and computed in the process of simultaneous degree reduction and knots removal
    (red dashed line).}
\end{center}
\end{figure}

\bibliographystyle{plain}


\end{document}